\documentclass{birkjour}
 \newtheorem{thm}{Theorem}[section]
 \newtheorem{cor}[thm]{Corollary}
 
 \newtheorem{prop}[thm]{Proposition}
 \newtheorem{notation}{Notation}
 \theoremstyle{definition}
 \newtheorem{defn}[thm]{Definition}
 \theoremstyle{remark}

 \numberwithin{equation}{section}

\begin{document}

%-------------------------------------------------------------------------
% editorial commands: to be inserted by the editorial office
%
%\firstpage{1} \volume{228} \Copyrightyear{2004} \DOI{003-0001}
%\seriesextraline{This is the Concrete Title of this Book\br H.E. R and S.T.C. W, Eds.}
%\firstpage{1}

\title[$AW(k)$-Type Curves According to the Bishop Frame]
 {$AW(k)$-Type Curves \\ According to the Bishop Frame}

%----------Author 1
\author[\.{I}lim K\.{I}\c{S}\.{I}]{\.{I}lim K\.{I}\c{S}\.{I}}

\address{Kocaeli University, Department of Mathematics, \\
41380, Kocaeli, Turkey}

\email{ilim.ayvaz@kocaeli.edu.tr}

%----------Author 2
\author{G\"{u}nay \"{O}ZT\"{U}RK}
\address{TKocaeli University, Department of Mathematics, \\
41380, Kocaeli, Turkey} \email{ogunay@kocaeli.edu.tr}


\date{January 1, 2004}

%%% ----------------------------------------------------------------------

\begin{abstract}
In this study, we consider $AW(k)$-type curves according to the
Bishop Frame in Euclidean space $\mathbb{E}^{3}$. We give the
relations between the Bishop curvatures $k_{1}$, $k_{2}$ of a curve
in $\mathbb{E}^{3}$.
\end{abstract}

%%% ----------------------------------------------------------------------
\maketitle
%%% ----------------------------------------------------------------------
%\tableofcontents
\section{Introduction}

The ability to "ride" along a three-dimensional space curve and
illustrate the properties of the curve, such as curvature and
torsion, would be a great asset to Mathematicians. The classic
Serret-Frenet frame provides such ability, however the Frenet-Serret
frame is not defined for all points along every curve. A new frame
is needed for the kind of Mathematical analysis that is typically
done with computer graphics. The Relatively Parallel Adapted Frame
or Bishop Frame could provide the desired means to ride along any
given space curve \cite{B}. The Bishop Frame has many properties
that make it ideal for mathematical research. Another area of
interested about the Bishop Frame is so-called Normal Development,
or the graph of the twisting motion of Bishop Frame. This
information along with the initial position and orientation of the
Bishop Frame provide all of the information necessary to define the
curve. The Bishop frame may have applications in the area of Biology
and Computer Graphics. For example it may be possible to compute
information about the shape of sequences of DNA using a curve
defined by the Bishop frame. The Bishop frame may also provide a new
way to control virtual cameras in computer animations. In \cite{KB}
authors studied natural curvatures of Bishop frame. In \cite{BK},
the same authors\ considered slant helix according to Bishop frame
in Euclidean $3$-Space. In \cite{UKT}, authors researched the spinor
formulations of curves according to Bishop frames in
$\mathbb{E}^{3}$.

The notion of $AW(k)$-type submanifolds was introduced by Arslan and
West in \cite{AW}. In particular, many works related to curves of
$AW(k)$-type have been done by several authors. For example, in
\cite{AO, OG}, the authors
gave curvature conditions and characterizations related to these curves in $%
\mathbb{E}^{m}$. In \cite{KA}, authors considered curves and surfaces of $%
AW(k)$ ($k=1,2$ or $3$)-type. They also give related examples of
curves and surfaces satisfying $AW(k)$-type conditions. In \cite{Y},
Yoon studied curves of $AW(k)$-type in the Lie group $G$ with a
bi-invariant metric. Also, he characterized general helices in terms
of $AW(k)$-type curve in the Lie group $G$.

This paper is organized as follows: Section $2$ gives some basic
concepts of
the Frenet \ frame and Bishop frame of a curve in $\mathbb{E}^{3}.$ Section $%
3$ tells the curvature conditions and characterizations of
$AW(k)$-type curve of osculating order $3.$ In the final section we
consider $AW(k)$-type curves and their curvatures $k_{1}$, $k_{2}$
according to the Bishop Frame in Euclidean space $\mathbb{E}^{3}$.

\section{Basic Concepts}

Let $\gamma =\gamma (s):I\rightarrow \mathbb{E}^{3}$ be arbitrary
curve in the Euclidean space $\mathbb{E}^{3}$, where $I$ is interval
in $\mathbb{R}$.
$\gamma $ is said to be of unit speed (parametrized by arc length function $%
s $) if $\left \Vert \gamma ^{\text{ }\prime }(s)\right \Vert =1.$
Then the
derivatives of the Frenet frame of $\gamma $ (Frenet-Serret formula);%
\begin{equation*}
\left[
\begin{array}{c}
T^{^{\prime }} \\
N^{^{\prime }} \\
B^{^{\prime }}%
\end{array}%
\right] =\left[
\begin{array}{ccc}
0 & \kappa & 0 \\
-\kappa & 0 & \tau \\
0 & -\tau & 0%
\end{array}%
\right] \left[
\begin{array}{c}
T \\
N \\
B%
\end{array}%
\right]
\end{equation*}%
where $\left \{ T,N,B\right \} $ is the Frenet frame of $\gamma $ and $%
\kappa $, $\tau $ are the curvature and torsion of curve $\gamma $,
respectively \cite{C}.

The Bishop frame or parallel transport frame is an alternative
approach to defining a moving frame that is well-defined even when
the curve has vanishing second derivative. One can express parallel
transport of an orthonormal frame along a curve simply by parallel
transporting each component of the frame. The tangent vector and any
convenient arbitrary basis for the remainder of the frame are used.
Therefore, the Bishop (frame)
formulas are expressed as%
\begin{equation*}
\left[
\begin{array}{c}
T~^{\prime } \\
M_{1}^{\prime } \\
M_{2}^{\prime }%
\end{array}%
\right] =\left[
\begin{array}{ccc}
0 & k_{1} & k_{2} \\
-k_{1} & 0 & 0 \\
-k_{2} & 0 & 0%
\end{array}%
\right] \left[
\begin{array}{c}
T \\
M_{1} \\
M_{2}%
\end{array}%
\right]
\end{equation*}%
where $\left \{ T,M_{1},M_{2}\right \} $ is the Bishop Frame and $k_{1}$, $%
k_{2}$ are called first and second Bishop curvatures, respectively
\cite{B}.

The relation between Frenet frame and Bishop frame is given as follows:%
\begin{equation*}
\left[
\begin{array}{c}
T \\
N \\
B%
\end{array}%
\right] =\left[
\begin{array}{ccc}
1 & 0 & 0 \\
0 & \cos \theta & \sin \theta \\
0 & -\sin \theta & \cos \theta%
\end{array}%
\right] \left[
\begin{array}{c}
T \\
M_{1} \\
M_{2}%
\end{array}%
\right]
\end{equation*}%
where $\theta (s)=\arctan \left( \frac{k_{2}}{k_{1}}\right) $, $\tau
(s)=\left( \frac{d\theta (s)}{ds}\right) $ and $\kappa (s)=\sqrt{%
k_{1}^{2}+k_{2}^{2}}$. Here Bishop curvatures are defined by
$k_{1}=\kappa \cos \theta ,$ $k_{2}=\kappa \sin \theta $.

\section{$AW(k)$-type Curves in $\mathbb{E}^{3}$}

In this section, we introduce Frenet curves of $AW(k)$-type.

\begin{prop}
\cite{AO} Let $\gamma $ be a Frenet curve of $\mathbb{E}^{3}$
osculating
order 3 then we have;%
\begin{equation*}
\gamma ^{\prime }\left( s\right) =T\left( s\right)
\end{equation*}%
\begin{equation*}
\gamma ^{\prime \prime }\left( s\right) =T^{\prime }\left( s\right)
=\kappa \left( s\right) N\left( s\right)
\end{equation*}%
\begin{equation*}
\gamma ^{\prime \prime \prime }\left( s\right) =-\kappa ^{2}\left(
s\right) T(s)+\kappa ^{\prime }\left( s\right) N\left( s\right)
+\kappa \left( s\right) \tau (s)B\left( s\right)
\end{equation*}%
\begin{eqnarray*}
\gamma ^{\imath v}\left( s\right) &=&-3\kappa (s)\kappa ^{\prime
}(s)T\left( s\right) +\left \{ \kappa ^{\prime \prime }(s)-\kappa
^{3}(s)-\kappa \left(
s\right) \tau ^{2}(s)\right \} N\left( s\right) \\
&&+\left \{ 2\kappa ^{\prime }(s)\tau (s)+\kappa \left( s\right)
\tau ^{\prime }(s)\right \} B\left( s\right) .
\end{eqnarray*}
\end{prop}

\begin{notation}
\cite{AO} Let us write%
\begin{equation*}
N_{1}(s)=\kappa \left( s\right) N(s)
\end{equation*}%
\begin{equation*}
N_{2}(s)=\kappa ^{\prime }\left( s\right) N\left( s\right) +\kappa
\left( s\right) \tau (s)B\left( s\right)
\end{equation*}%
\begin{equation*}
N_{3}(s)=\left \{ \kappa ^{\prime \prime }(s)-\kappa ^{3}(s)-\kappa
\left( s\right) \tau ^{2}(s)\right \} N\left( s\right) +\left \{
2\kappa ^{\prime }(s)\tau (s)+\kappa \left( s\right) \tau ^{\prime
}(s)\right \} B\left( s\right) .
\end{equation*}
\end{notation}

\begin{defn}
\cite{AO} Frenet curves (of osculating order 3) are

$i)$ of type weak $AW(2)$ if they satisfy%
\begin{equation*}
N_{3}\left( s\right) =\left \langle N_{3}\left( s\right)
,N_{2}^{\ast }\left( s\right) \right \rangle N_{2}^{\ast }\left(
s\right)
\end{equation*}

$ii)$ of type weak $AW(3)$ if they satisfy%
\begin{equation*}
N_{3}\left( s\right) =\left \langle N_{3}\left( s\right)
,N_{1}^{\ast }\left( s\right) \right \rangle N_{1}^{\ast }\left(
s\right) .
\end{equation*}
\end{defn}

where
\begin{equation*}
N_{1}^{\ast }\left( s\right) =\frac{N_{1}\left( s\right) }{\left
\Vert N_{1}\left( s\right) \right \Vert },
\end{equation*}%
\begin{equation*}
N_{2}^{\ast }\left( s\right) =\frac{N_{2}\left( s\right) -\left
\langle N_{2}\left( s\right) ,N_{1}^{\ast }\left( s\right) \right
\rangle N_{1}^{\ast }\left( s\right) }{\left \Vert N_{2}\left(
s\right) -\left \langle N_{2}\left( s\right) ,N_{1}^{\ast }\left(
s\right) \right \rangle N_{1}^{\ast }\left( s\right) \right \Vert }.
\end{equation*}

\begin{prop}
\cite{AO} Let $\gamma $ be a Frenet curve of order 3. Then $\gamma $
is a weak $AW(2)$ type curve if and only if
\begin{equation*}
\kappa ^{\prime \prime }(s)-\kappa ^{3}(s)-\kappa \left( s\right)
\tau ^{2}(s)=0.
\end{equation*}
\end{prop}

\begin{defn}
\cite{AO} Frenet curves of order 3\ are

\ $i)$ of type $AW(1)$ if they satisfy%
\begin{equation*}
N_{3}\left( s\right) =0
\end{equation*}%
\ \ \ \ \ \ \ \ \ \ \ \

$ii)$\ of type $AW(2)$ if they satisfy%
\begin{equation*}
\left \Vert N_{2}\right \Vert ^{2}N_{3}=\left \langle
N_{3},N_{2}\right \rangle N_{2}\ \
\end{equation*}%
\ \ \ \ \

$iii)$ of type $AW(3)$ if they satisfy%
\begin{equation*}
\left \Vert N_{1}\right \Vert ^{2}N_{3}=\left \langle
N_{3},N_{1}\right \rangle N_{1}.
\end{equation*}
\end{defn}

\begin{prop}
\cite{AO} Let $\gamma $ be a Frenet curve of order 3. Then $\gamma $
is of
type $AW(1)$ if and only if%
\begin{equation*}
\kappa ^{\prime \prime }(s)-\kappa ^{3}(s)-\kappa \left( s\right)
\tau ^{2}(s)=0
\end{equation*}%
and%
\begin{equation*}
\tau \left( s\right) =\frac{c}{\kappa ^{2}\left( s\right) },c\in \mathbb{R}%
\text{.}
\end{equation*}
\end{prop}

\begin{prop}
\cite{AO} Let $\gamma $ be a Frenet curve of order 3. Then $\gamma $
is of
type $AW(2)$ if and only if%
\begin{equation*}
2(\kappa ^{\prime }\left( s\right) )^{2}\tau \left( s\right) +\kappa
\left( s\right) \kappa ^{\prime }\left( s\right) \tau ^{\prime
}\left( s\right) =\kappa \left( s\right) \kappa ^{\prime \prime
}\left( s\right) \tau \left( s\right) -\kappa ^{4}\left( s\right)
\tau \left( s\right) -\kappa ^{2}\left( s\right) \tau ^{3}\left(
s\right) .
\end{equation*}
\end{prop}

\begin{prop}
\cite{AO} Let $\gamma $ be a Frenet curve of order 3. Then $\gamma $
is of
type $AW(3)$ if and only if%
\begin{equation*}
2\kappa ^{\prime }\left( s\right) \tau \left( s\right) +\kappa
\left( s\right) \tau ^{\prime }\left( s\right) =0.
\end{equation*}%
The solution of this differential equation is $\tau \left( s\right) =\frac{c%
}{\kappa ^{2}\left( s\right) },c\in \mathbb{R}$.
\end{prop}

\section{$AW(k)$-type Curves According to the Bishop Frame}

In this section, we consider $AW(k)$-type curves according to the
Bishop Frame in Euclidean space $\mathbb{E}^{3}$.

\begin{prop}
Let $\gamma $ be a unit speed curve of osculating order 3 in
$\mathbb{E}^{3}$ and $\left \{ T,M_{1},M_{2}\right \} $ be Bishop
frame of the curve. Similar
with proposition 1 in \cite{AO}, we have%
\begin{equation*}
\gamma ^{\prime }\left( s\right) =T\left( s\right)
\end{equation*}%
\begin{equation*}
\gamma ^{\prime \prime }\left( s\right) =k_{1}\left( s\right)
M_{1}\left( s\right) +k_{2}\left( s\right) M_{2}\left( s\right)
\end{equation*}%
\begin{equation*}
\gamma ^{\prime \prime \prime }\left( s\right) =\left \{
-k_{1}^{2}\left( s\right) -k_{2}^{2}\left( s\right) \right \}
T\left( s\right) +k_{1}^{\prime }\left( s\right) M_{1}\left(
s\right) +k_{2}^{\prime }\left( s\right) M_{2}\left( s\right)
\end{equation*}%
\begin{eqnarray*}
\gamma ^{\imath v}\left( s\right) &=&\left \{ \left(
-k_{1}^{2}\left( s\right) -k_{2}^{2}\left( s\right) \right) ^{\prime
}+\left( -k_{1}\left( s\right) k_{1}^{\prime }\left( s\right)
-k_{2}\left( s\right) k_{2}^{\prime
}\left( s\right) \right) \right \} T\left( s\right) \\
&&+\left \{ k_{1}^{\prime \prime }\left( s\right) -k_{1}^{3}\left(
s\right) -k_{1}\left( s\right) k_{2}^{2}\left( s\right) \right \}
M_{1}\left( s\right)
\\
&&+\left \{ k_{2}^{\prime \prime }\left( s\right) -k_{2}^{3}\left(
s\right) -k_{1}^{2}\left( s\right) k_{2}\left( s\right) \right \}
M_{2}\left( s\right)
\end{eqnarray*}
\end{prop}

\begin{notation}
Let us write%
\begin{equation}
\overline{N}_{1}\left( s\right) =k_{1}\left( s\right) M_{1}\left(
s\right) +k_{2}\left( s\right) M_{2}\left( s\right)  \label{a2}
\end{equation}%
\begin{equation}
\overline{N}_{2}\left( s\right) =k_{1}^{\prime }\left( s\right)
M_{1}\left( s\right) +k_{2}^{\prime }\left( s\right) M_{2}\left(
s\right)  \label{a3}
\end{equation}%
\begin{eqnarray}
\overline{N}_{3}\left( s\right) &=&\left \{ k_{1}^{\prime \prime
}\left( s\right) -k_{1}^{3}\left( s\right) -k_{1}\left( s\right)
k_{2}^{2}\left(
s\right) \right \} M_{1}\left( s\right)  \label{a4} \\
&&+\left \{ k_{2}^{\prime \prime }\left( s\right) -k_{2}^{3}\left(
s\right) -k_{1}^{2}\left( s\right) k_{2}\left( s\right) \right \}
M_{2}\left( s\right) .  \notag
\end{eqnarray}
\end{notation}
\begin{cor}
$\gamma ^{\prime }\left( s\right) ,\gamma ^{\prime \prime }\left(
s\right) ,\gamma ^{\prime \prime \prime }\left( s\right) $ and
$\gamma ^{\imath
v}\left( s\right) $ are linearly dependent if and only if $\overline{N}%
_{1}\left( s\right) $, $\overline{N}_{2}\left( s\right) $ and $\overline{N}%
_{3}\left( s\right) $ are linearly dependent.
\end{cor}
\begin{thm}
Let $\gamma $ be a curve of order 3, then%
\begin{equation}
\overline{N}_{3}\left( s\right) =\left \langle
\overline{N}_{3}\left(
s\right) ,\overline{N}_{1}^{\ast }\left( s\right) \right \rangle \overline{N}%
_{1}^{\ast }\left( s\right) +\left \langle \overline{N}_{3}\left( s\right) ,%
\overline{N}_{2}^{\ast }\left( s\right) \right \rangle \overline{N}%
_{2}^{\ast }\left( s\right)  \label{a5}
\end{equation}%
where
\begin{equation}
\overline{N}_{1}^{\ast }\left( s\right)
=\frac{\overline{N}_{1}\left( s\right) }{\left \Vert
\overline{N}_{1}\left( s\right) \right \Vert }, \label{a6}
\end{equation}%
\begin{equation}
\overline{N}_{2}^{\ast }\left( s\right)
=\frac{\overline{N}_{2}\left(
s\right) -\left \langle \overline{N}_{2}\left( s\right) ,\overline{N}%
_{1}^{\ast }\left( s\right) \right \rangle \overline{N}_{1}^{\ast
}\left( s\right) }{\left \Vert \overline{N}_{2}\left( s\right)
-\left \langle \overline{N}_{2}\left( s\right)
,\overline{N}_{1}^{\ast }\left( s\right) \right \rangle
\overline{N}_{1}^{\ast }\left( s\right) \right \Vert }. \label{a7}
\end{equation}
\end{thm}

\begin{defn}
The unit speed curves of order 3\ are

$i)$ of type weak Bishop $AW(2)$ if they satisfy%
\begin{equation}
\overline{N}_{3}\left( s\right) =\left \langle
\overline{N}_{3}\left(
s\right) ,\overline{N}_{2}^{\ast }\left( s\right) \right \rangle \overline{N}%
_{2}^{\ast }\left( s\right)  \label{a8*}
\end{equation}

$ii)$ of type weak Bishop $AW(3)$ if they satisfy%
\begin{equation}
\overline{N}_{3}\left( s\right) =\left\langle \overline{N}_{3}\left(
s\right) ,\overline{N}_{1}^{\ast }\left( s\right) \right\rangle \overline{N}%
_{1}^{\ast }\left( s\right) .  \label{a9}
\end{equation}
\end{defn}

\begin{prop}
\ Let $\gamma $ be a curve of order 3. If $\gamma $ is of type weak Bishop $%
AW(2)$ then the equations%
\begin{equation}
\begin{array}{l}
\left( k_{1}^{\prime \prime }\left( s\right) -k_{1}^{3}\left(
s\right) -k_{1}\left( s\right) k_{2}^{2}\left( s\right) \right)
\left(
k_{1}^{2}\left( s\right) +k_{2}^{2}\left( s\right) \right) \\
=k_{2}^{2}\left( s\right) \left( k_{1}^{\prime \prime }\left(
s\right) -k_{1}^{3}\left( s\right) -k_{1}\left( s\right)
k_{2}^{2}\left( s\right)
\right) \\
-k_{1}\left( s\right) k_{2}\left( s\right) \left( k_{2}^{\prime
\prime }\left( s\right) -k_{2}^{3}\left( s\right) -k_{1}^{2}\left(
s\right)
k_{2}\left( s\right) \right) ,%
\end{array}
\label{a10}
\end{equation}%
\begin{equation}
\begin{array}{l}
\left( \text{$k_{2}^{\prime \prime }\left( s\right) -k_{2}^{3}\left(
s\right) -k_{1}^{2}\left( s\right) k_{2}\left( s\right) $}\right)
\left(
k_{1}^{2}\left( s\right) +k_{2}^{2}\left( s\right) \right) \\
=k_{1}^{2}\left( s\right) \left( k_{2}^{\prime \prime }\left(
s\right) -k_{2}^{3}\left( s\right) -k_{1}^{2}\left( s\right)
k_{2}\left( s\right)
\right) \\
-k_{1}\left( s\right) k_{2}\left( s\right) \left( k_{1}^{\prime
\prime }\left( s\right) -k_{1}^{3}\left( s\right) -k_{1}\left(
s\right)
k_{2}^{2}\left( s\right) \right)%
\end{array}
\label{a11}
\end{equation}

are hold and a solution of these equations are%
\begin{equation}
k_{1}\left( s\right) =-k_{2}\left( s\right) =\pm
\frac{1}{s+c};c=cons. \label{a12}
\end{equation}
\end{prop}

\begin{proof}
With the use of equations (\ref{a2}), (\ref{a3}), (\ref{a6}) and
(\ref{a7}), we get
\begin{equation}
\overline{N}_{1}^{\ast }\left( s\right)
=\frac{1}{\sqrt{k_{1}^{2}\left( s\right) +k_{2}^{2}\left( s\right)
}}\left( k_{1}\left( s\right) M_{1}\left( s\right) +k_{2}\left(
s\right) M_{2}\left( s\right) \right)   \label{b1}
\end{equation}%
\ and
\begin{equation}
\overline{N}_{2}^{\ast }\left( s\right)
=\frac{1}{\sqrt{k_{1}^{2}\left( s\right) +k_{2}^{2}\left( s\right)
}}\left( k_{2}\left( s\right) M_{1}\left( s\right) -k_{1}\left(
s\right) M_{2}\left( s\right) \right) .  \label{b2}
\end{equation}%
Substituting equation (\ref{b2}) in (\ref{a8*}), we obtain (\ref{a10}) and (%
\ref{a11}). From these equations, we get $k_{1}\left( s\right)
=-k_{2}\left( s\right) $. Finally, writing this in equation
(\ref{a10}) we get the result.
\end{proof}

\begin{prop}
Let $\gamma $ be a curve of order 3. If $\gamma $ is of type weak Bishop $%
AW(3)$ then the curvature equations%
\begin{equation*}
\begin{array}{l}
\left( k_{1}^{\prime \prime }\left( s\right) -k_{1}^{3}\left(
s\right) -k_{1}\left( s\right) k_{2}^{2}\left( s\right) \right)
\left(
k_{1}^{2}\left( s\right) +k_{2}^{2}\left( s\right) \right) \\
=(k_{1}^{2}\left( s\right) \left( k_{1}^{\prime \prime }\left(
s\right) -k_{1}^{3}\left( s\right) -k_{1}\left( s\right)
k_{2}^{2}\left( s\right)
\right) \\
+k_{1}\left( s\right) k_{2}\left( s\right) \left( k_{2}^{\prime
\prime }\left( s\right) -k_{2}^{3}\left( s\right) -k_{1}^{2}\left(
s\right)
k_{2}\left( s\right) \right)%
\end{array}%
\end{equation*}

and%
\begin{equation*}
\begin{array}{l}
\left( k_{2}^{\prime \prime }\left( s\right) -k_{2}^{3}\left(
s\right) -k_{1}^{2}\left( s\right) k_{2}\left( s\right) \right)
\left(
k_{1}^{2}\left( s\right) +k_{2}^{2}\left( s\right) \right) \\
=k_{2}^{2}\left( s\right) \left( k_{2}^{\prime \prime }\left(
s\right) -k_{2}^{3}\left( s\right) -k_{1}^{2}\left( s\right)
k_{2}\left( s\right)
\right) \\
+k_{1}\left( s\right) k_{2}\left( s\right) \left( k_{1}^{\prime
\prime }\left( s\right) -k_{1}^{3}\left( s\right) -k_{1}\left(
s\right)
k_{2}^{2}\left( s\right) \right) .%
\end{array}%
\end{equation*}

are hold.
\end{prop}

\begin{proof}
Substituting the equation (\ref{b1}) in (\ref{a9}) we get the
result.
\end{proof}

\begin{defn}
The unit speed curves of order 3\ are

$\ i)$ of type Bishop $AW(1)$ if they satisfy%
\begin{equation}
\overline{N}_{3}\left( s\right) =0  \label{a15}
\end{equation}%
\ \ \ \ \ \ \ \ \ \ \ \

$ii)$\ of type Bishop $AW(2)$ if they satisfy%
\begin{equation}
\left \Vert \overline{N}_{2}\right \Vert ^{2}\overline{N}_{3}=\left
\langle \overline{N}_{3},\overline{N}_{2}\right \rangle
\overline{N}_{2}  \label{a16}
\end{equation}%
\ \ \ \ \ \ \

$iii)$ of type Bishop $AW(3)$ if they satisfy%
\begin{equation}
\left \Vert \overline{N}_{1}\right \Vert ^{2}\overline{N}_{3}=\left
\langle \overline{N}_{3},\overline{N}_{1}\right \rangle
\overline{N}_{1}. \label{a17}
\end{equation}
\end{defn}

\begin{prop}
Let $\gamma $ be a curve of order 3. If $\gamma $ is of type Bishop
$AW(1)$ then
\begin{equation}
k_{1}^{\prime \prime }\left( s\right) -k_{1}^{3}\left( s\right)
-k_{1}\left( s\right) k_{2}^{2}\left( s\right) =0,  \label{a18}
\end{equation}%
\begin{equation}
k_{2}^{\prime \prime }\left( s\right) -k_{2}^{3}\left( s\right)
-k_{1}^{2}\left( s\right) k_{2}\left( s\right) =0  \label{a19}
\end{equation}%
and a solution of this differential equation system is
\begin{equation}
k_{1}\left( s\right) =k_{2}\left( s\right) =\pm \frac{1}{s+c},\text{
}c\in \mathbb{R}.  \label{a20}
\end{equation}
\end{prop}

\begin{proof}
With the use of equations (\ref{a4}) and (\ref{a15}), we get the
differential system (\ref{a18}) and (\ref{a19}). By solving this
system, we
get $k_{1}\left( s\right) =k_{2}\left( s\right) $ and using this in (\ref%
{a18}) or (\ref{a19}), we get the result.
\end{proof}

\begin{prop}
Let $\gamma $ be a curve of order 3. If $\gamma $ is of type Bishop
$AW(2)$
then%
\begin{equation}
k_{2}^{\prime }\left( s\right) \left( k_{1}^{\prime \prime }\left(
s\right) -k_{1}^{3}\left( s\right) -k_{1}\left( s\right)
k_{2}^{2}\left( s\right) \right) =k_{1}^{\prime }\left( s\right)
\left( k_{2}^{\prime \prime }\left( s\right) -k_{2}^{3}\left(
s\right) -k_{1}^{2}\left( s\right) k_{1}\left( s\right) \right) .
\label{a21}
\end{equation}
\end{prop}

\begin{proof}
Substituting equations (\ref{a3}) and (\ref{a4}) in (\ref{a16}), we
get the result.
\end{proof}

\begin{prop}
Let $\gamma $ be a curve of order 3. If $\gamma $ is of type Bishop
$AW(3)$
then%
\begin{equation}
k_{2}\left( s\right) \left( k_{1}^{\prime \prime }\left( s\right)
-k_{1}^{3}\left( s\right) -k_{1}\left( s\right) k_{2}^{2}\left(
s\right) \right) =k_{1}\left( s\right) \left( k_{2}^{\prime \prime
}\left( s\right) -k_{2}^{3}\left( s\right) -k_{1}^{2}\left( s\right)
k_{1}\left( s\right) \right) .  \label{a22}
\end{equation}
\end{prop}

\begin{proof}
Substituting equations (\ref{a2}) and (\ref{a4}) in (\ref{a17}), we
get the result.
\end{proof}

\end{document}